\documentclass[11pt]{article}

\usepackage{graphicx}
\usepackage{cite}
\usepackage{color}
\usepackage{amsfonts}
\usepackage{amsthm}
\usepackage{url}
\usepackage{amsmath}
\usepackage{amssymb}
\usepackage{enumerate}
\usepackage{subfig}
\usepackage{fixltx2e}
\usepackage{dsfont}
\usepackage{pseudocode}
\usepackage{mathrsfs}

\newtheorem{theorem}{Theorem}
\newtheorem{lemma}{Lemma}
\newtheorem{remark}{Remark}

\newtheorem{corollary}{Corollary}

\newtheorem{example}{Example}
\newcommand{\EE}{\mathbb{E}}

\renewcommand{\d}{d}

\newcommand{\Var}{\operatorname{Var}}

\renewcommand{\to}{\longrightarrow}
\usepackage{eufrak}
\graphicspath{{figs/}}
\usepackage[numbers]{natbib}

\usepackage{fullpage}
\title{Bounds on the Poincar\'e constant for convolution measures}
\author{Thomas~A.~Courtade\\University of California, Berkeley}
\date{June 29, 2018}
\begin{document}

\maketitle

\begin{abstract}
We establish a Shearer-type inequality for the Poincar\'e constant, showing that the Poincar\'e constant corresponding to the convolution of a collection of measures can be nontrivially controlled  by the  Poincar\'e constants corresponding to convolutions of subsets of measures.    This implies, for example, that the Poincar\'e constant is non-increasing along the central limit theorem.   We also establish a dimension-free  stability estimate for subadditivity of the Poincar\'e constant  on convolutions which uniformly improves  an earlier one-dimensional estimate of a similar nature by Johnson (2004).  As a byproduct of our arguments, we find that the monotone properties of entropy, Fisher information and the Poincar\'e constant along the CLT find a common, simple root in Shearer's inequality. 
\end{abstract}

\section{Introduction}
Let $\mathcal{P}(\mathbb{R}^d)$ denote the set of Borel probability measures on $\mathbb{R}^d$.  
A  measure $\mu \in \mathcal{P}(\mathbb{R}^d)$ is said to satisfy a Poincar\'e inequality with constant $C$ if 
\begin{align}
\Var_{\mu}(f) \leq C \int_{\mathbb{R}^d} |\nabla f|^2 d\mu \label{PI}
\end{align}
 for all locally Lipschitz functions $f: \mathbb{R}^d\to \mathbb{R}$, where $\nabla$ denotes the usual gradient and $|\cdot|$ denotes the Euclidean length on $\mathbb{R}^d$.  The Poincar\'e constant $C_p(\mu)$ is defined to be the smallest constant $C$ for which \eqref{PI} holds.  
 
Poincar\'e inequalities play a central role in concentration of measure (see, e.g., \cite[Ch. 3]{ledoux2001concentration}), and imply dimension-free concentration inequalities for the product measures $\mu^n$, $n\geq 1$, which depend only on the Poincar\'e constant $C_p(\mu)$.  Indeed, it is an easy exercise to see that $C_p(\mu^n) = C_p(\mu)$, so the Poincar\'e inequality directly implies 
\begin{align*}
\Var_{\mu^n}(f) \leq  C_p(\mu) \|f \|^2_{\mathrm{Lip}}  ~~~~~~~~\forall f: \mathbb{R}^{nd} \to \mathbb{R},
\end{align*}
 where $ \|\cdot\|_{\mathrm{Lip}}$ is the usual Lipschitz seminorm.  Stronger concentration estimates are also available.  For example, Bobkov and Ledoux \cite{bobkov1997poincare} established  the following dimension-free estimate for exponential concentration 
 \begin{align}
 \mu^n\left( f \geq \int f d\mu^n + t\right) \leq \exp\left(-\min\left(\frac{t^2}{K C_p(\mu) }, \frac{t}{\sqrt{K C_p(\mu)} } \right) \right),\label{eq:dimFreeExpConcentration}
 \end{align}
 holding for all 1-Lipschitz $f$,  where $K$ is an absolute constant.   A  converse statement also holds, implying that dimension-free concentration is equivalent to existence of a Poincar\'e inequality in a precise sense \cite{gozlan2015dimension}. Thus, any information about $C_p(\mu)$ reveals quantitative information about the concentration properties enjoyed by $\mu$.   Beyond  concentration of measure, Poincar\'e inequalities  play an important role throughout analysis, for example in characterizing convergence of stochastic dynamics.

Except in special cases,  $C_p(\mu)$ is not known explicitly for general probability measures $\mu$, but it can sometimes be controlled using properties enjoyed by the Poincar\'e constant. For example, it is easy to check by change of variables that $C_p(\mu_{\alpha,\beta}) = \alpha^2 C_{p}(\mu)$, where $\mu_{\alpha,\beta}\sim \alpha X + \beta$, with $\alpha, \beta \in \mathbb{R}$ and $X\sim \mu$.    Only slightly less immediate is the subadditivity property 
\begin{align}
C_p(\mu\star\nu)\leq C_p(\mu) + C_p(\nu) \label{eq:ClassicalSubadditivity}
\end{align}
for the convolution measure $\mu\star\nu$, which follows by a classical variance decomposition and convexity of $t\mapsto t^2$ (e.g., \cite{borovkov1983inequality}).  

 It is convolution inequalities like \eqref{eq:ClassicalSubadditivity} that are the focus of this paper.   There have been several  recent results along these lines which we now mention.  For example, Bardet,   Gozlan,   Malrieu  and Zitt \cite{bardet2018functional} recently established dimension-free bounds on the Poincar\'e constant for {G}aussian convolutions of compactly supported measures.  Johnson  \cite{johnson2004convergence} had obtained similar bounds on the Poincar\'e constant for finite mixtures of one-dimensional Gaussians with identical variances, and  he further studied the convergence of the Poincar\'e constant along the central limit theorem.  This latter topic is closely related to some of the results contained in this paper, so we highlight the similarities and differences in the relevant sections.   In a related direction, Chafai and  Malrieu \cite{chafai2010fine} gave bounds on the Poincar\'e constant for two-point mixtures.  We remark that other bounds are known for the logarithmic Sobolev constant for convolution measures (which immediately yield bounds on the Poincar\'e constant), e.g. \cite{zimmermann2013logarithmic, bardet2018functional}, but these tend to be weaker than estimates which   target the Poincar\'e constant directly. 
 
Our results fall into two main categories.  First, we establish a Shearer-type inequality for the Poincar\'e constant, which shows that the Poincar\'e constant corresponding to the convolution of a collection of measures can be nontrivially controlled  by the Poincar\'e constants corresponding to convolutions of subsets of measures.    This has new and interesting consequences; for example, the Poincar\'e constant is monotone along the central limit theorem, similar to entropy.   Second, we establish a dimension-free quantitative stability estimate for the inequality \eqref{eq:ClassicalSubadditivity}, which depends on the measures $\mu, \nu$ only through their Poincar\'e constants.  This uniformly improves upon a previous estimate of Johnson  \cite{johnson2004convergence}, which  required a bound on Fisher information.  The proofs all rely on a particular variance inequality, closely related to Shearer's lemma, which may be of independent interest.  As a byproduct of our arguments, we see that the monotone property of entropy, Fisher information and the Poincar\'e constant along the CLT find a common root in Shearer's inequality.

\section{Presentation of Results}

\subsection{Bounds on the Poincar\'e constant for convolution measures}

Our first  result is the following bound on the Poincar\'e constant for convolutions:
\begin{theorem}\label{thm:monotoneSubsets}
Let $(\mu_i)_{1\leq i\leq n} \subset \mathcal{P}(\mathbb{R}^d)$.  For a   set $S\subset[n] :=\{1,2,\dots, n\}$, let $\mu_S$ denote the convolution of $(\mu_i)_{i\in S}$.  
If   $\mathcal{C}$ is a collection of distinct subsets of $[n]$, then 
\begin{align}
 C_p(\mu_{[n]}) \leq \frac{1}{t} \sum_{S \in \mathcal{C}} C_p(\mu_S) \label{eq:ConvolutionSubset}
\end{align}
where $t := \min_{i \in [n]} \#\{S \in \mathcal{C} : S\ni i\}$. 
\end{theorem}
As an example, if we take $n=2$ and $\mathcal{C} =\{ \{1\},\{2\}\}$, we see that \eqref{eq:ConvolutionSubset}   extends the classical subadditivity estimate \eqref{eq:ClassicalSubadditivity}.   Two further examples in the case of convolutions of identical measures --- where the expressions are simplest,  but the results still new and illustrative --- are given below:  %
\begin{example}\label{ex:monotoneCLT}
Let $(X_i)$ be i.i.d.~random vectors in $\mathbb{R}^d$ with law $\nu_1$.   For $n\geq 1$, let $\nu_n$ denote the law of the standardized sum $\frac{1}{\sqrt{n}} \sum_{i=1}^n X_i$.  Then, 
\begin{align}
C_p(\nu_n) \leq C_p(\nu_{n-1}).\label{CpCLT}
\end{align}
That is, the Poincar\'e constant is non-increasing along the central limit theorem. 
\end{example}
Like entropy and Fisher information, monotonicity of  $C_p(\nu_n)$ along the central limit theorem is suggested by the classical subaddivity property (though, not implied by it).  %
 The connection between \eqref{CpCLT} and monotonicity of entropy and Fisher information runs deep, and  is articulated in Section \ref{sec:varianceInequality} (Remark \ref{rmk:VarianceDrop}).
 
The following example is more counterintuitive, and does not seem to be predicted by subadditivity:
\begin{example}\label{ex:decreaseReg}
Let  $\nu_n$ be as in the previous example, and let $\gamma_{\delta^2}$ denote the law of the normal distribution $N(0,\delta^2 \mathrm{I})$.  Then
\begin{align}
C_p(\nu_n \star \gamma_{\delta^2/n})\leq C_p(\nu_1 \star \gamma_{\delta^2}).\label{decreaseReg}
\end{align}
The surprise here is that the degree of gaussian regularization on the left is significantly less than that on the right (i.e., variance $\delta^2/n$ instead of $\delta^2$), but the Poincar\'e constant is no worse. %
\end{example}

To add to Example \ref{ex:decreaseReg}, we remark that  discrete probability distributions  are typical examples of measures that do not satisfy a Poincar\'e inequality, since one can always find  a non-constant function $f$ which remains constant on the support, ensuring $\nabla f=0$ on sets of positive measure.  In Example \ref{ex:decreaseReg}, we can take $\nu_1$ to be the equiprobable measure on $\{-\tfrac{1}{2},\tfrac{1}{2}\}$, and $\delta^2\ll1$.  In this case, the measure $\nu_n \star \gamma_{\delta^2/n}$ looks nearly discrete as $n$ becomes fairly large (i.e., like a standardized Binomial$(n,1/2)$ distribution), yet satisfies a Poincar\'e inequality with constant depending only on $\delta^2$, e.g., $C_p(\nu_n \star \gamma_{\delta^2/n})  \leq C_p(\nu_n \star \gamma_{\delta^2/n}) \leq \delta^2 \exp(4 /\delta^2)$.   To contrast this concrete example with previous estimates, an application of \cite[Theorem 1.2]{bardet2018functional} gives $C_p(\nu_n \star \gamma_{\delta^2/n}) \leq \frac{\delta^2}{n} \exp(4 n^2 /\delta^2)$, and Johnson's estimate for univariate Gaussian mixtures  \cite[Theorem 1.4]{johnson2004convergence} gives $C_p(\nu_n \star \gamma_{\delta^2/n}) \leq 2^n \exp( n 2^n/\delta^2)$.   These bounds may be improved slightly  using subadditivity, but  nevertheless remain exponential in $n$. 

\subsubsection{Application to a quantitative CLT in $W_2$}
The ability to introduce vanishing regularization via \eqref{decreaseReg} may be useful in applications.  One nice illustration is the following dimension-free quantitative central limit theorem in  the $L^2$-Wasserstein distance on $\mathcal{P}(\mathbb{R}^d)$, denoted by $W_2$:
\begin{corollary}\label{cor:w2clt}
Let $(X_i)$ be i.i.d. centered isotropic random vectors in $\mathbb{R}^d$ with law $\nu_1$.  If $C_p(\nu_1\star \gamma_{\delta^2}) \leq C_{\delta^2}$ for some $\delta^2>0$, then 
\begin{align}
W_2(\nu_n,\gamma_1)^2 \leq d \frac{2(\delta^2 + C_{\delta^2})}{  \delta^2+\sqrt{n}-1} , \label{W2CLT}
\end{align}
where $\nu_n \sim \frac{1}{\sqrt{n}}\sum_{i=1}^n X_i$. 
\end{corollary}

\begin{proof}
The proof involves properties of the so-called Stein discrepancy $S(\mu|\gamma_1)^2$, which is a measure of the distance from a probability measure $\mu$ to the standard gaussian measure $\gamma_1$.  Its precise definition is not needed here, but we will need two properties.  First, for any centered probability measure $\mu\in \mathcal{P}(\mathbb{R}^d)$, finiteness of $C_p(\mu)$ implies finiteness of $S(\mu|\gamma_1)^2$ (see \cite{courtade2017existence}).  Combined with results from \cite{ledoux2015stein}, if $\mu$ is centered and isotropic, then 
$$
W_2(\mu,\gamma_1)^2 \leq S(\mu|\gamma_1)^2 \leq (C_p(\mu)-1) d
$$
and, with the notation  $\nu_n$ prevailing, 
$$
S(\nu_{n+1}|\gamma_1)^2 \leq S(\nu_n|\gamma_1)^2 \leq \frac{1}{n}S(\nu_1|\gamma_1)^2 ~~~~\mbox{for all $n\geq 1$}.
$$

For $n\geq 1$, let $\nu^t_{n}\sim \frac{n^{-1/2}}{\sqrt{1+t}}\sum_{i=1}^n X_i$, so that $\nu^t_{n} \star\gamma_{t/(1+t)}$ is isotropic.
Starting with the triangle inequality for $W_2$ and using each of the above estimates followed by \eqref{decreaseReg}, we have  for any $t>0$ and integers $n_1,n_2\geq 0$ such that $ n_1 n_2\leq n$, 

\begin{align*}
\frac{1}{2}W_2(\nu_n, \gamma_1)^2 &\leq W_2(\nu^t_{n} \star\gamma_{t/(1+t)}, \nu_n)^2 +W_2(\nu^t_{n} \star\gamma_{t/(1+t)}, \gamma_1)^2 \\
&\leq  d \frac{t}{1+t} + S(\nu^t_{n} \star\gamma_{t/(1+t)} | \gamma_1)^2\\
&\leq d \frac{t}{1+t} + \frac{1}{n_1}S(\nu^t_{n_2} \star\gamma_{t/(1+t)} | \gamma_1)^2\\
&\leq d  \frac{t}{1+t} + \frac{d}{n_1} C_p(\nu^t_{n_2} \star\gamma_{t/(1+t)})  \\
&= d  \frac{t}{1+t} + \frac{d}{n_1}\Big( \frac{1}{1+t}C_p(\nu_{n_2} \star\gamma_{t}) \Big) \\
&\leq d  \frac{t}{1+t} + \frac{d}{n_1}\Big( \frac{1}{1+t}C_p(\nu_{1} \star\gamma_{n_2t}) \Big).%
\end{align*} 
Now,  choosing $t = \delta^2/n_2$ and $n_1 =n_2 = \lfloor \sqrt{n}\rfloor \geq \sqrt{n}-1$, we simplify to find \eqref{W2CLT}.
\end{proof}

A few remarks are in order:  
We emphasize that for \eqref{W2CLT} to hold, $\nu_1$ does not need to satisfy a Poincar\'e inequality, but only needs to have finite Poincar\'e constant after convolution with a gaussian of sufficiently large variance\footnote{In fact, if $\nu_1$ has $C_p(\nu_1)<\infty$, then we can take $\delta=0$, $n_1=n$ and $n_2=1$ in the the proof of Corollary \ref{cor:w2clt} to conclude $W_2(\nu_n, \gamma_1)^2=O(d C_p(\nu_1)/n)$, which can be found in \cite{courtade2017existence}. 
}.  This is a significantly weaker assumption than finiteness of the Poincar\'e constant; for instance, a simple modification of the proof of \cite[Theorem 1.2]{bardet2018functional} establishes that all subguassian distributions enjoy this property. %
Since $C_p(\mu^n) = C_p(\mu)$, the estimate \eqref{W2CLT} is dimension-free, and in fact has optimal dependence on dimension since $W_2^2$ is additive on product measures.  However, the rate $O(n^{-1/2})$ is suboptimal, and should be $O(1/n)$ under moment constraints.  In particular, Talagrand's inequality together with results of Bobkov, Chistyakov and G\"otze \cite{bobkov2014berry} imply  an asymptotic rate of $W_2(\nu_n,\gamma_1)^2 =O(1/n)$ under finite fourth moment, but these estimates are non-quantitative in dimension greater than one, and prefactors generally behave poorly in dimension. 

In \cite{zhai2018high}, Zhai improved an earlier result by Valiant and Valiant \cite{valiant2011estimating} and established that if $\nu_1$ is supported in the Euclidean ball of radius $R$, then it holds that
\begin{align}
W_2(\nu_n,\gamma_1)^2 \leq   \frac{25 d R^2 (1+\log n)^2}{ n}. \label{zhaiCLT}
\end{align}
Since any isotropic measure supported in the Euclidean ball of radius $R$ necessarily has $R^2\geq d$, Zhai's estimate gives at best $O(d^2)$ scaling in the upper bound on $W_2^2$, which is  worse than  \eqref{W2CLT}.  However, \eqref{zhaiCLT} does offer the improved rate of $O((\log n)^2/n )$ compared to the rate of $O(n^{-1/2})$ in \eqref{W2CLT}.  As a result, if one is working with compactly supported distributions, then  \eqref{W2CLT} would be preferred in the sample-limited regime where $n/ (\log n)^4 \lesssim d^2$, and \eqref{zhaiCLT} would be preferred in the large-sample regime where $n/ (\log n)^4 \gtrsim d^2$.  

We also remark that near-optimal rates have very recently been obtained for the multivariate CLT in the weaker 1-Wasserstein distance under the assumption of finite third moments, albeit with suboptimal dependence on dimension \cite{gallouet2018regularity}.

\subsubsection{Remark on non-Euclidean settings}
Poincar\'e inequalities continue to make sense in settings beyond $\mathbb{R}^d$, but the applications remain similar (see, e.g., \cite{gozlan2010poincare}).   For example, if $(\mathcal{X},d)$ is a metric space equipped with a probability measure $\mu$, it is common to say $\mu$ satisfies a Poincar\'e inequality with constant $C$ if 
\begin{align}
\Var_{\mu}(f)\leq C \int_{\mathcal{X}} |\nabla f|^2 d\mu \label{genPoincare}
\end{align}
for a sufficiently large class of test functions $f:\mathcal{X}\to \mathbb{R}$.  Here,  the length of the gradient is defined as 
\begin{align}
|\nabla f|(x) :=\limsup_{y\to x} \frac{|f(x)-f(y)|}{d(x,y)} \label{metricDerivative}
\end{align}
whenever $x$ is an accumulation point (otherwise $|\nabla f|(x) =0$).  Under this definition, the dimension-free concentration estimate \eqref{eq:dimFreeExpConcentration} continues to hold \cite{bobkov1997poincare}.  %

It turns out that Theorem \ref{thm:monotoneSubsets} can be extended to cover these and other situations.    
Let $(\mathcal{X},+)$ be an abelian group, where $\mathcal{X}$ is a Polish space, and let $\mathbb{B}(\mathcal{X},\mathbb{R})$ denote the collection of bounded real-valued functions on $\mathcal{X}$.  Consider a collection of   functions $\mathcal{A} \subset \mathbb{B}(\mathcal{X},\mathbb{R})$  which is closed under translation; i.e.,  $f \in \mathcal{A} \Leftrightarrow f(\cdot +t) \in \mathcal{A}$ for all $t\in \mathcal{X}$.  Further, let $\nabla$ be an operator on the elements of $\mathcal{A}$ which commutes with translation in the sense that   $|\nabla ( f(\cdot + t))|  = |(\nabla f)(\cdot +t)|$ for all $f\in \mathcal{A}$ and $t\in \mathcal{X}$.  In words, the length of the gradient of the map $x\mapsto f(x+t)$ is equal to the length of the gradient of $f$, evaluated at $x+t$.   Note that this condition is met for \eqref{metricDerivative}, assuming the metric $d$ is translation invariant.   It is also satisfied by discrete derivatives in symmetric settings (e.g., the hypercube). 

With the above definitions, for a given probability measure $\mu\in \mathcal{P}(\mathcal{X})$, define $C_p(\mu ; \mathcal{A})$ to be the smallest constant $C$ such that \eqref{genPoincare} holds for all $f\in \mathcal{A}$.  
\begin{theorem} \label{Thm:nonEuclidean}
Let the above notation prevail, and consider a collection of probability measures  $(\mu_i) \subset\mathcal{P}(\mathcal{X})$.  For $S\subset[n]$, let $\mu_S$ denote the law of $\sum_{i\in S} X_i$, where $X_i\sim \mu_i$ are independent and summation is with respect to the group operation $+$.  If   $\mathcal{C}$ is a collection of distinct subsets of $[n]$, then 
 $$
C_p(\mu_{[n]} ; \mathcal{A}) \leq \frac{1}{t} \sum_{S\in \mathcal{C}} C_p(\mu_{S} ; \mathcal{A})
$$
where $t := \min_{i \in [n]} \#\{S \in \mathcal{C} : S\ni i\}$. 
\end{theorem}
In applications, $\mathcal{A}$ will generally be dense in the class of test functions with respect to an appropriate norm.  For example, in the case of $\mathcal{X}=\mathbb{R}^d$ where $\nabla$ is the usual (weak) gradient, then Theorem \ref{thm:monotoneSubsets} follows from Theorem \ref{Thm:nonEuclidean} by density of smooth functions in the Sobolev space $W^{1,2}(\mathbb{R}^d,\mu)$, where $\mu$ has density with respect to Lebesgue measure.

\subsection{Stability of subadditivity of the  Poincar\'e constant}

For $\mu\in \mathcal{P}(\mathbb{R}^d)$, define 
$$
\sigma^2(\mu) := \max_{\alpha \in \mathbb{R}^d: |\alpha|=1} \Var_{\mu}(x\mapsto \alpha \cdot x)
$$
to be the largest variance of $\mu$ in any direction (equivalently, the largest eigenvalue of the covariance matrix for $\mu$).    It is known that $C_p(\mu)\geq \sigma^2(\mu)$, with equality only if $\mu$ is marginally gaussian in the direction of largest variance.  In fact, as shown in \cite{courtade2017existence},  if $\alpha^* \in \arg \max_{\alpha \in \mathbb{R}^d: |\alpha|=1} \Var_{\mu}(x\mapsto \alpha \cdot x)$, then 
\begin{align}
C_p(\mu) - \sigma^2(\mu)  \geq   W_2\Big( (\alpha^* \cdot \operatorname{id}) \# \mu, \gamma_{\sigma^2(\mu)}\Big)^2,\label{eq:W2bound}
\end{align}
where $\gamma_{\sigma^2}$ is the law of $N(0,\sigma^2(\mu))$, and $(\alpha^* \cdot \operatorname{id}) \# \mu$ is the pushforward of $\mu$ under the map $x\mapsto \alpha^* \cdot x$  (i.e.,  $(\alpha^* \cdot \operatorname{id}) \# \mu$ is the marginal distribution of $\mu$ in direction $\alpha^*$).   We remark that that the one-dimensional nature of \eqref{eq:W2bound} is unavoidable, since the Poincar\'e constant of $\mu$ is at least as bad as any one-dimensional marginal (e.g., consider product measures of the form $\gamma_{\sigma^2} \otimes \mu^{ n}$). 

Our main result of this section is a dimension-free stability estimate for the subadditivity property \eqref{eq:ClassicalSubadditivity} in terms of the gap $C_p(\mu\star\nu)-\sigma^2(\mu\star\nu)$ (and, by \eqref{eq:W2bound}, in terms of the non-gaussianness of any one-dimensional marginal of $\mu$ with largest variance). 
In particular, 
\begin{theorem}\label{thm:Cstable}
Let $\mu, \nu\in \mathcal{P}(\mathbb{R}^d)$, and define $\sigma^2 = \sigma^2(\mu\star\nu)$ for convenience. 
Then, 
$$
C_p(\mu\star \nu) \leq (C_p(\mu) + C_p(\nu) )    - \frac{C_p(\mu) C_p(\nu)}{C_p(\mu) + C_p(\nu)}  \frac{(C_p(\mu\star \nu) - \sigma^2)^2}{(C_p(\mu\star \nu) - \sigma^2)^2 + C_p(\mu\star \nu) \sigma^2}.$$
\end{theorem}
Of note, the above estimate is dimension-free, and requires no quantitative information about the measures beyond their Poincar\'e constants and the value $\sigma^2(\mu\star\nu)$.  In order for equality to hold in \eqref{eq:ClassicalSubadditivity}, we need that $C_p(\mu\star \nu) = \sigma^2(\mu\star\nu)$, implying that $\mu\star \nu$ is marginally gaussian in its direction of maximum variance by \eqref{eq:W2bound}.  Cramer's theorem then implies that $\mu$ and $\nu$ must be marginally gaussian in this same direction as well.

Letting $\mu = \nu$ in Theorem \ref{thm:Cstable}, the following stability estimate for i.i.d. sums is immediate:
\begin{corollary}\label{cor:IIDstable}
Let $X_1,X_2$ be i.i.d.~random vectors in $\mathbb{R}^d$ with law $\nu_1$, and define $\nu_2$ to be the law of the standardized sum $\frac{1}{\sqrt{2}}(X_1 + X_2)$.  Then 
\begin{align}
C_p(\nu_2) \leq   C_p(\nu_1)    - \frac{C_p(\nu_1) }{4}  \frac{(C_p(\nu_2) - \sigma^2)^2}{(C_p(\nu_2) - \sigma^2)^2 + C_p(\nu_2)\sigma^2},\label{CstableIID}
\end{align}
where $\sigma^2 =\sigma^2 (\nu_2) = \sigma^2 (\nu_1)$.
\end{corollary}

Johnson's paper contains a result  similar to Corollary \ref{cor:IIDstable}, so we  describe  notable differences below.  First, we mention that our results hold for probability measures on $\mathbb{R}^d$, whereas Johnson's results are derived only for  dimension 1.  Perhaps more substantially, the stability estimate derived by Johnson depends on Fisher information, whereas ours does not.  Assuming $d=\sigma^2(\nu_1)=1$, Johnson's key result can be stated in a form comparable to \eqref{CstableIID} as\footnote{This result is not stated explicitly, but can be distilled from Eq. (3) of \cite{johnson2004convergence}.}
\begin{align}
C_p(\nu_2) \leq   C_p(\nu_1)    - \frac{C_p(\nu_1) }{9}  \frac{(C_p(\nu_2) - 1)^2}{C_p(\nu_2)^2 (1+ J(\nu_1) C_p(\nu_1)) },
\label{JohnsonStable}
\end{align}
where 
\begin{align*}
J(\mu) := \int_{\mathbb{R}} \frac{ |f'(x)|^2}{f(x)}dx 
\end{align*}
denotes the Fisher information associated to a probability measure $\mu$ on $\mathbb{R}$ with differentiable density $\d\mu(x) = f(x) dx$.   Since $J(\nu_1)\geq 1$ in this setting (by the Cramer-Rao inequality), we see that \eqref{CstableIID} always improves upon \eqref{JohnsonStable}, and the improvement can be significant when either $J(\nu_1)$ or  $C_p(\nu_1)$ is large.  As a particular example, our results apply to uniform measures on convex sets (a prototypical class of measures with finite Poincar\'e constant), whereas \eqref{JohnsonStable} degenerates to subaddivity   since Fisher information is infinite   due to discontinuity of the density at the boundary of its support (in fact, even under arbitrarily small  regularization, the Fisher information will still tend to infinity).   We additionally note that  Fisher information is additive on product measures, so a 
na\"ive extension of \eqref{JohnsonStable} to $\mathbb{R}^d$ would seem to suggest a stability estimate that degrades quickly with  dimension.

Johnson's motivation for establishing \eqref{JohnsonStable} was to quantify convergence of the Poincar\'e constants $C_p(\nu_n)$ along the CLT, where $\nu_n$ is the same as in Example \ref{ex:monotoneCLT}.  In particular, the main result of \cite{johnson2004convergence} claims for   $d=1$ and $\sigma^2(\nu_1)=1$, 
$$
C_p(\nu_n) \leq 1 + \frac{c}{n}, 
$$
where $c$ is a constant depending only on $C_p(\nu_1)$ and $J(\nu_1)$.    In actuality, however, this rate of convergence is  established  along the subsequence $(\nu_{2^n})$, rather than $(\nu_n)$ as desired.  That is, Theorem 1.2  of \cite{johnson2004convergence} should instead state 
\begin{align}
C_p(\nu_{n}) \leq 1 + \frac{c}{\log n}, \label{JohnsonDecay}
\end{align}
giving an effective rate of convergence of $O(1/\log n)$, rather than $O(n^{-1})$. The mistake appears to be due to a notational oversight in going from the proof of the theorem to the statement of the theorem itself, rather than a technical error.    %

In view of this, we take the opportunity to revisit the topic of convergence of the Poincar\'e constant.   Unfortunately, despite the improvements of \eqref{CstableIID} over \eqref{JohnsonStable}  described above, our Corollary \ref{cor:IIDstable} seems incapable of showing that $C_p(\nu_n)-1$ decays asymptotically better than $O(1/\log n)$, whereas a rate of $O(1/n)$ would naturally be conjectured as  optimal. Although we suffer from this shortcoming in the asymptotic regime, we can positively show that the Poincar\'e constant for standardized sums of random vectors converges  quickly to a universal constant (i.e., 3/2), before the slower convergence rate kicks in.  The precise statement is as follows:

\begin{theorem}\label{thm:CpConverge}
Let $X_1, X_2, \dots$ be i.i.d.~isotropic random vectors in $\mathbb{R}^d$, and define $\nu_n$ to be the law of the standardized sum $\frac{1}{\sqrt{n}}\sum_{i=1}^n X_i$.  
Then  
\begin{align}
C_p(\nu_{2^n})  -\frac{3}{2} \leq   \left(3 \over 4\right)^n \left(C_p(\nu_1)-\frac{3}{2}\right) \label{eq:CpConverge1}
\end{align}
and, if $C_p(\nu_{1}) \leq 2$, it further holds that 
\begin{align}
 C_p(\nu_{2^n}) -1 \leq  \frac{7}{n+7}. \label{eq:CpConverge2}
\end{align}
\end{theorem}
\begin{remark}
Using the monotone property $C_p(\nu_{n+1})\leq C_p(\nu_{n})$ established in \eqref{ex:monotoneCLT}, the above inequalities together give an explicit bound on the convergence of the sequence $C_p(\nu_{n})\searrow 1$, depending only on the initial Poincar\'e constant $C_p(\nu_1)$. %
\end{remark}
\begin{remark}
The i.i.d.~assumption on the sequence $(X_i)$ may be relaxed to independence with uniformly bounded Poincar\'e constants. 
\end{remark}

\section{Proofs of main results}

\subsection{A variance inequality}\label{sec:varianceInequality}
This section is  devoted to the proof of a particular variance inequality, from which our main results will follow.  First, recall that for a measurable space $\mathcal{X}$, and two probability measures $P\ll Q$ on $\mathcal{X}$, the relative entropy between $P$ and $Q$ is defined as
$$
D(P\|Q) = \int_{\mathcal{X}}\log\left( \frac{dP}{dQ}\right) dP.
$$

Our starting point is a projection-type inequality enjoyed by relative entropy known as Shearer's lemma (finding origins in \cite{chung1986some}), which generalizes inequalities due to Han  \cite{te1978nonnegative}.    Before stating the result, we establish some notation. Let $(\mathcal{X}_i, d_i)$, $i=1, \dots, n$, be a collection separable complete metric spaces.  If $P$ is a probability measure on the product space $\mathcal{X} = \prod_{i=1}^n\mathcal{X}_i$, let $P_{S}$ denote the  corresponding marginal distribution on $\mathcal{X}_S : = \prod_{i\in S}\mathcal{X}_i$, where $S\subset\{1,\dots, n\}$.  That is, $P_S = \pi_S \# P$, where $\pi_S : \mathcal{X} \to \mathcal{X}_S$ is the natural projection.   With this notation, Shearer's lemma is the following:
\begin{theorem} 
Let $P, Q$ be  Borel probability measures on $\mathcal{X}$, where $Q$ has product form $Q = \prod_{i=1}^n Q_i$. %
For any collection $\mathcal{C}$ of distinct subsets of $\{1, \dots, n\}$, 
\begin{align}
\sum_{S\in \mathcal{C} }  D(P_{S} \| Q_{S})  \leq r D(P\|Q),\label{shearerIneq}
\end{align}
where $r := \max_{i} \#\{ S\in \mathcal{C} : S\ni i\}$. 
\end{theorem}
The  assumption that $\mathcal{X}$ is a product of separable complete metric spaces  ensures that the disintegration theorem  can be applied to $P$, which is needed for the proof.  This assumption is more than sufficient for our purposes, where we consider only the case where $\mathcal{X}_i = \mathbb{R}^d$.  Aside from this technical point, the proof is straightforward  and is included below for completeness.  
\begin{proof}
For an integer $k\geq 1$, define $[k] = \{1, \dots, k\}$, and $S^{k} = S\cap [k]$ for $S\subset[n]$.    Further, for $S,T\subset\{1,\dots, n\}$, let $P_{S|T}(\cdot | x)$ denote the conditional distribution of $P$ on $\mathcal{X}_S$, given $x \in \mathcal{X}_T$.   With notation established, the proof is a simple consequence of properties of relative entropy (cf. \cite{cover2012elements}):
\begin{align*}
\sum_{S\in \mathcal{C} }  D(P_{S} \| Q_{S})   
&= \sum_{S\in \mathcal{C}}   \sum_{i\in S}  \int_{\mathcal{X}_{S^{i-1}}} 
D\Big(P_{i | {S^{i-1}}}(\cdot|s) \Big\|Q_{i}(\cdot) \Big)  dP_{S^{i-1}}(s) \\
&\leq  \sum_{S\in \mathcal{C} }   \sum_{i\in S} \int_{\mathcal{X}_{[i-1]}} 
D\Big(P_{i | [i-1] }(\cdot|s) \Big\|Q_{i}(\cdot) \Big) d P_{ [i-1]} (s) \\
&=  \sum_{i=1}^n \sum_{S\in \mathcal{C}  : S\ni i }   \int_{\mathcal{X}_{[i-1]}}  
D\Big(P_{i | [i-1]}(\cdot|s) \Big\|Q_{i}(\cdot) \Big)  dP_{ [i-1]} (s)\\
&\leq   r \sum_{i=1}^n    \int_{\mathcal{X}_{[i-1]}}  D\Big(P_{i | [i-1]}(\cdot|s) \Big\|Q_{i}(\cdot) \Big)  dP_{ [i-1]} (s) \\
&=r D(P\|Q).
\end{align*}
The first inequality is due to convexity of $D$, and the second follows from the definition of $r$.
\end{proof}
The key inequality we shall need in the present paper is the following:
\begin{corollary}
With notation as above, let $Q$ be a  probability measure on $\mathcal{X}$ with product form, and for the random vector  $X = (X_1, \dots, X_n)\sim Q$, let $X_S = (X_i)_{i\in S}$ be the natural projection of $X$ onto $\mathcal{X}_S$. For any collection $\mathcal{C}$ of distinct subsets of $\{1, \dots, n\}$, and any $f:\mathcal{X}\to \mathbb{R}$, 
\begin{align}
\sum_{S \in \mathcal{C}}
\Var \left( \EE[f(X) | X_S ] \right) \leq r \Var( f(X)  ), \label{varianceprojection}
\end{align}
where the (conditional) expectation is with respect to $Q$ and $r := \max_{i\in[n]} \#\{ S\in \mathcal{C} : S\ni i\}$. 
\end{corollary}
\begin{proof}
The proof follows by linearizing \eqref{shearerIneq}. 
We may assume $f :\mathcal{X}\to  \mathbb{R}$ is bounded with $\int f d Q = 0$; the general claim follows by density.   For sufficiently small $\epsilon$, define the probability measure $P$ via $d P = (1+\epsilon f) dQ$.  Now,  apply  \eqref{shearerIneq} and take Taylor series about $\epsilon =0$ to conclude \eqref{varianceprojection}. 
\end{proof}

Despite the fundamental nature of \eqref{varianceprojection}, we could not find any explicit appearance of it in the literature, though it may already be known to some.  In fact, as pointed out by Y.~Polyanskiy, Shearer's lemma holds for any non-negative submodular set function (however, the easiest way to verify the hypothesis for the set function $S\mapsto \Var \left( \EE[f(X) | X_S ] \right)$ may be through a linearization argument applied to entropy, as above). 
 In any case, we are aware of a few related results, which we now briefly discuss. 

\begin{remark}
A simple modification gives the following:  Let $X_1, \dots, X_n$ be mutually independent random vectors in $\mathbb{R}^d$, and define the random sums $U_S = \sum_{i\in S}X_i$.  In this case, for any collection $\mathcal{C}$ of distinct subsets of $\{1, \dots, n\}$ and $f:\mathbb{R}^d\to\mathbb{R}$, 
\begin{align}
\sum_{S \in \mathcal{C}}
\Var \left( \EE[f(U_{[n]}) | U_S] \right) \leq r \Var ( f(U_{[n]})  ),\notag
\end{align}
where, as before, $r := \max_{i\in[n]} 
\#\{ S\in \mathcal{C} : S\ni i\}$.  In particular, if $(X_i)$ are i.i.d., then a simple consequence is the inequality
\begin{align}
\Var \left( \EE[f(U_{[n]}) | U_{[m]}] \right) \leq \frac{m}{n} \Var ( f(U_{[n]})  ),   ~~~~~1\leq m\leq n,\notag
\end{align}
which is equivalent to the main result of Dembo, Kagan and Shepp in \cite{dembo2001remarks}. 
 \end{remark}

\begin{remark}\label{rmk:VarianceDrop}
As mentioned in the remarks following Example \ref{ex:monotoneCLT},  monotonicity of the Poincar\'e constant along the CLT parallels the same property enjoyed by entropy and Fisher information, first proved in \cite{artstein2004solution}.  In fact, the subset inequality of Theorem \ref{thm:monotoneSubsets} reminds one of similar subset inequalities enjoyed by entropy and Fisher information, which were proved by Madiman and Barron \cite{madiman2007generalized}.   This relationship is not coincidental, and we explain the connection here.  Specifically, the critical estimate needed in \cite{madiman2007generalized} is a ``variance drop" inequality of the form
\begin{align}
 \EE \left| \sum_{S\in \mathcal{C} } \psi_S(X_S) \right|^2 \leq  r \sum_{S\in \mathcal{C} }  \EE \Big|\psi_S(X_S)\Big|^2, \label{MB:variancedrop}
\end{align}
where the notation $X_S$ is the same as above for independent $(X_i)_{i\in [n]}$,  and $\psi_S :\mathcal{X}_S\to \mathbb{R}$, $S\in \mathcal{C}$,  are any  functions satisfying  $\EE \psi_S(X_S) = 0$ for each $S\in \mathcal{C}$. Madiman and Barron proved  \eqref{MB:variancedrop} using ANOVA decompositions, and apply it to monotonicity of entropy/Fisher information by setting $(\psi_S)_{S\in\mathcal{C}}$ to be score functions of partial sums.    As they noted in their paper, \eqref{MB:variancedrop} generalizes a classical result on $U$-statistics due to Hoeffding \cite{hoeffding1948class}.

Since \eqref{varianceprojection} plays a central role in the proof of Theorem \ref{thm:monotoneSubsets}, the connection between monotonicity of Poincar\'e constants  and  Fisher information/entropy along the CLT can  be realized through the connection between \eqref{varianceprojection} and \eqref{MB:variancedrop}.  In particular, a new proof of \eqref{MB:variancedrop} can be obtained from \eqref{varianceprojection} as follows:  Identifying $f(x) := \sum_{S\in \mathcal{C}} \psi_S(x_S)$ and applying Cauchy-Schwarz twice followed by \eqref{varianceprojection}, we have
\begin{align*}
  \EE \left| \sum_{S\in \mathcal{C} } \psi_S(X_S) \right|^2   &= 
 \frac{ \left(  \sum_{S\in \mathcal{C} } \EE f(X)  \psi_S(X_S) \right)^2 }{    \EE |f(X) |^2   }\\
 &\leq    \frac{ \left(  \sum_{S\in \mathcal{C} }  \left(  \EE | \EE[ f(X)|X_S]|^2 \right)^{1/2} \left( \EE |\psi_S(X_S)|^2 \right)^{1/2} \right)^2}{  \EE |f(X) |^2    } \\
 &\leq \left(   \frac{  {\sum_{S\in \mathcal{C} } }  \EE | \EE[ f(X)|X_S]|^2  }{    \EE |f(X) |^2    } \right) \left(\sum_{S\in \mathcal{C} }  \EE |\psi_S(X_S)|^2 \right)  \\
 &\leq   r \sum_{S\in \mathcal{C} }  \EE \Big|\psi_S(X_S)\Big|^2.
\end{align*}
Hence, the monotonicity results for entropy, Fisher information and the Poincar\'e constant are seen to have a common root in Shearer's inequality. 
\end{remark}

With \eqref{varianceprojection} in hand, the proof of Theorem \ref{thm:monotoneSubsets} now follows readily. 

\begin{proof}[Proof of Theorem \ref{thm:monotoneSubsets}]
For independent random vectors $X_i\sim \mu_i$, $i\in [n]$,   and   $S\subset [n]$, define $X_S = (X_i)_{i\in S}$, and let $U = \sum_{i=1}^n X_i$. 
Consider any smooth $f : \mathbb{R}^d\to \mathbb{R}$.  For any $S\subset [n]$, we have the classical variance decomposition
\begin{align*}
\Var( f ( U ) ) =  \EE[ \Var( f(U) | X_S ) ] + \Var\left( \EE[   f(U) |X_S] \right). 
\end{align*}
Summing over subsets $S\in \mathcal{C}$ and applying \eqref{varianceprojection}, we find 
\begin{align*}
|\mathcal{C}| \Var( f ( U ) ) &=  \sum_{S\in \mathcal{C}} \EE[ \Var( f(U) | X_S ) ] + \sum_{S\in \mathcal{C}} \Var\left( \EE[   f(U) |X_S] \right)\\
&\leq  \sum_{S\in \mathcal{C}} \EE[ \Var( f(U) | X_S ) ] + r  \Var\left(  f(U) \right),
\end{align*}
where $r := \max_{i \in [n]} \#\{ S\in \mathcal{C} : S\ni i\}$.   Rearranging and applying the Poincar\'e inequality for $\mu_{[n]\setminus S}$, $S\in \mathcal{C}$, we have 
\begin{align*}
(|\mathcal{C}|-r) \Var( f ( U ) ) \leq  \sum_{S\in \mathcal{C}} \EE[ \Var( f(U) | X_S ) ]  &\leq  \sum_{S\in \mathcal{C}}   \EE\Big[ C_p(\mu_{[n]\setminus S}) \EE\left[ |\nabla f(U)|^2| X_S\right]  \Big]\\
&=  \sum_{S\in \mathcal{C}}  C_p(\mu_{[n]\setminus S})   \EE\left[ |\nabla f(U)|^2 \right].
\end{align*}
Now, the proof is complete by relabeling sets $S\leftarrow [n]\setminus S$ since $
(|\mathcal{C}|-r) =  \min_{i \in [n]} \#\{S \in \mathcal{C} : ([n]\setminus S)\ni i\}$.
\end{proof}

\begin{remark}
The proof of Theorem \ref{Thm:nonEuclidean} is identical, and is therefore omitted. 
\end{remark}

\subsection{Proof of Theorems \ref{thm:Cstable} and \ref{thm:CpConverge}}
This section is dedicated to the proofs of Theorems  \ref{thm:Cstable} and \ref{thm:CpConverge}.  We begin with a technical lemma.    First, a remark on notation throughout this section: 

\begin{remark}  For a vector-valued function $g : \mathcal{X} \to \mathbb{R}^d$ and a probability measure $\mu$ on $\mathcal{X}$, we  abuse notation slightly and write $\Var_{\mu}(g)$ to denote $\sum_{i=1}^d \Var_{\mu}(g_i)$, where $g_i: \mathcal{X} \to \mathbb{R}^d$ denotes the $i$th coordinate of $g = (g_1, \dots, g_d)$.  In particular, 
$$
\Var_{\mu}(g) := \int |g|^2 \d\mu - \left| \int g \d\mu \right|^2.
$$
\end{remark}
\begin{lemma}\label{lem:VarBound}
Let $\mu$ be a probability measure on $\mathbb{R}^d$ which verifies a Poincar\'e inequality with best constant $C_p=C_p(\mu)$. Define $\sigma^2 := \sup_{\alpha : |\alpha| =1} \Var_{\mu}(x\mapsto \alpha\cdot x)$.  %
 There is a sequence $(f_n)$ of  real-valued functions on $\mathbb{R}^d$ with  $\int |\nabla f_n|^2d\mu = 1$, which satisfies 
$$%
\lim_{n\to\infty} \Var_{\mu}( f_n)   = C_p,
$$%
and 
$$
\lim_{n\to\infty}\Var_{\mu}(\nabla f_n ) \geq \frac{(C_p - \sigma^2)^2}{(C_p - \sigma^2)^2 + C_p \sigma^2}.%
$$
\end{lemma}
\begin{remark}
The idea here is that if $C_p > \sigma^2$, then there are near-extremizers of the Poincar\'e inequality for $\mu$ which have  nontrivial projection onto the space of nonlinear functions.  As a result, the variances $\Var_{\mu}(\nabla f_n )$ can be nontrivially compared to the moments $\int |\nabla f_n|^2 \d\mu$.  This result is suggested by Johnson for dimension 1 in \cite{johnson2004convergence}, but is only formally argued under the assumption that an extremizer exists for the Poincar\'e inequality.  The potential  nonexistence of extremizers is the main issue to be dealt with, and can be handled through an  application of the Lax-Milgram theorem, as  below. 
\end{remark}
\begin{proof}
We first show that  if $f$ is sufficiently smooth, satisfying   $\int f d\mu = 0$ and 
\begin{align}
\int |f|^2 d\mu  \geq (1-\epsilon^2)C_p \int |\nabla f|^2 \d\mu,  \label{fNearExtremal}
\end{align}
then 
\begin{align}
C_p \int \nabla f \cdot \nabla h \d\mu - \int f h d\mu 
\leq C_p \epsilon \left(\int |  \nabla f|^2d\mu \right)^{1/2}  \left(\int |\nabla h|^2 d\mu\right)^{1/2} \label{prelimInequality1}
\end{align}
for all sufficiently smooth $h$.   This may be seen as a stable form of the Euler-Lagrange equation associated to the Poincar\'e inequality for $\mu$.   Indeed, if there exists a nonzero function $f_0$ which $\Var_{\mu}(f_0) = C_p \int |\nabla f_0|^2
d\mu$, then 
$$
C_p \int \nabla f_0 \cdot \nabla h \d\mu = \int f_0 h d\mu 
$$
for all sufficiently smooth $h$. 

Toward establishing \eqref{prelimInequality1}, consider the Sobolev space $W^{1,2}$, defined as the closure of  the set of  functions $\{ f \in C^{\infty}(\mathbb{R}^d) : \int f d\mu = 0\}$ in $L^2(\mu)$ with respect to the Sobolev norm $\|f\| := \left( \int |\nabla f|^2 d\mu + \int |f|^2 d\mu\right)^{1/2}$.

By the Poincar\'e inequality, the continuous bilinear map $(f,g) \mapsto C_p \int \nabla f \cdot \nabla g d\mu$ is coercive on $W^{1,2}\times W^{1,2}$.  So, for any $f\in W^{1,2}$, the Lax-Milgram theorem ensures the existence of $u_f\in W^{1,2}$ such that
$$
C_p \int \nabla u_f \cdot \nabla h d\mu = \int f h d\mu ~~~~~~\mbox{for all~} h \in W^{1,2}.
$$
Applying this to the function $h = u_f$, Cauchy-Schwarz gives
$$
C_p \int |\nabla u_f|^2 d\mu =  \int u_f f d\mu \leq \left( C_p \int |\nabla u_f |^2 \d\mu \right)^{1/2}\left( \int | f |^2 \d\mu \right)^{1/2},
$$
so that $C_p \int |\nabla u_f|^2 d\mu  \leq   \int | f |^2 \d\mu$.  Now, if $f$ verifies \eqref{fNearExtremal}, then 
\begin{align*}
C_p \int |\nabla u_f - \nabla f|^2d\mu &=C_p \int |\nabla u_f  |^2d\mu +C_p \int |  \nabla f|^2d\mu -  2 C_p \int \nabla u_f \cdot  \nabla f d\mu \\
&=C_p \int |\nabla u_f  |^2d\mu +C_p \int |  \nabla f|^2d\mu -  2   \int | f|^2 d\mu \\
&\leq  \epsilon^2 C_p \int |  \nabla f|^2d\mu.
\end{align*}
As a consequence, we obtain 
\begin{align*}
C_p \int \nabla f \cdot \nabla h \d\mu - \int f h d\mu  &= C_p \int (\nabla f-\nabla u_f ) \cdot \nabla h d\mu\\
&\leq C_p \epsilon \left(\int |  \nabla f|^2d\mu \right)^{1/2}\left(\int |\nabla h|^2 d\mu\right)^{1/2},
\end{align*}
which is \eqref{prelimInequality1}.
 
Henceforth, we assume $f$ satisfies  $\int f d\mu = 0$ and \eqref{fNearExtremal}.   We may also assume without loss of generality that $\int x d\mu(x) =0$, since translation does not change the Poincar\'e constant.  For any $\alpha \in \mathbb{R}^d$,   definition of $\sigma^2$ together with the Poincar\'e  and Cauchy-Schwarz inequalities yields
\begin{align*}
\int f(x)  (\alpha \cdot x)  d\mu(x) -  |\alpha |^2 \sigma^2 
&= \int (f(x)-\alpha \cdot x)(\alpha \cdot x)d\mu(x)\\
&\leq \left( \int |f(x)-\alpha \cdot x|^2 d\mu(x)    \right)^{1/2}\left( \int |\alpha \cdot x|^2 d\mu(x)    \right)^{1/2}\\
&\leq  |\alpha| \sigma  \left( C_p \int |\nabla f -\alpha |^2 d\mu    \right)^{1/2}.%
\end{align*}
Applying \eqref{prelimInequality1} with $h(x) = \alpha \cdot x$, we conclude 
\begin{align}
C_p \int \nabla f \cdot \alpha d\mu(x) -  |\alpha |^2 \sigma^2 
&\leq |\alpha| \left( C_p \int |\nabla f -\alpha |^2 d\mu    \right)^{1/2} \sigma + C_p \epsilon \left(\int |  \nabla f|^2d\mu \right)^{1/2}  |\alpha|. \notag
\end{align}
Specializing this by taking $\alpha = \int \nabla f d\mu$, we find upon rearranging that 
\begin{align}
&\left(\int |  \nabla f|^2d\mu   -   \Var_{\mu}(\nabla f)  \right)^{1/2}   \left(  \sqrt{\frac{C_p}{\sigma^2}}  -     \sqrt{\frac{\sigma^2}{C_p}}  \right)\notag\\
&\leq   \left(  \Var_{\mu}(\nabla f)  \right)^{1/2}   + \epsilon \sqrt{\frac{C_p}{\sigma^2}} \left(\int |  \nabla f|^2d\mu     \right)^{1/2}  . \label{prelimInequality2}
\end{align}

At this point, the proof is essentially complete.  Indeed, by homogeneity of the Poincar\'e inequality, we can find a sequence $(f_n)\subset W^{1,2}$ such that  $\int |\nabla f_n|^2d\mu = 1$ and
$$%
\lim_{n\to\infty} \Var_{\mu}( f_n)  = C_p.
$$%
Since  $\Var_{\mu}(\nabla f_n ) \leq \int |\nabla f_n|^2d\mu = 1$, we may extract a subsequence for which the limit $\lim_{k\to\infty}\Var_{\mu}(\nabla f_{n_k} ) =:V_{\infty}$ exists.  Applying \eqref{prelimInequality2} to this subsequence, we may let $\epsilon\downarrow 0$ to conclude
 \begin{align}
\left(1   -   V_{\infty}  \right)^{1/2}   \left(  \sqrt{\frac{C_p}{\sigma^2}}  -     \sqrt{\frac{\sigma^2}{C_p}}  \right)
&\leq   \left( V_{\infty} \right)^{1/2}     .\notag
\end{align}
Squaring both sides and rearranging completes the proof. 
\end{proof}

Now, with the help of  Lemma \ref{lem:VarBound} and \eqref{varianceprojection}, we are ready to prove Theorem \ref{thm:Cstable}.

\begin{proof}[Proof of Theorem \ref{thm:Cstable}]
For convenience, we use probabilistic notation with $X\sim \mu$ and $Y\sim \nu$ independent, and $U = X+Y$. With this notation, we first aim to show that 
 \begin{align}
\Var(f(U)) &\leq  
(C_p(\mu) + C_p(\nu) )  \EE\left[|\nabla f(U)|^2 \right]  - \frac{C_p(\mu) C_p(\nu)}{C_p(\mu) + C_p(\nu)}  \Var(\nabla f(U)) \label{needToEliminateVariance}
\end{align}
for differentiable $f$.  To this end, we  consider a smooth test function $f : \mathbb{R}^d \to \mathbb{R}$; the general result  follows from density.   Without loss of generality, we may also assume   $\EE[f(U)]=0$.   We have the classical variance decomposition
\begin{align*}
\Var (f(U)) &= \EE \Var ( f(U)|X  )  + \Var\left( \EE[f(U)|X] \right)\\
&=: A + B.
\end{align*}

As in the proof of Theorem \ref{thm:monotoneSubsets}, since $\nu$ satisfies a Poincar\'e inequality with constant $C_p(\nu)$, the first term $A$ is bounded by 
\begin{align*}
A  &=\EE \Var ( f(U)|X  ) \leq \EE\left[ C_p(\nu) \EE[ |\nabla f(U)|^2 | X] \right]  = C_p(\nu) \EE\left[ |\nabla f(U)|^2  \right].  %
\end{align*}
Departing from the proof of Theorem \ref{thm:monotoneSubsets}, we bound the second term $B$ as
\begin{align*}
B &= \Var\left( \EE[f(U)|X] \right) \leq C_p(\mu) \EE \left| \nabla \EE[f(U)|X] \right|^2 = C_p(\mu) \EE \left|  \EE[ \nabla f(U)|X] \right|^2, %
\end{align*}
where moving the gradient inside the expectation is justified by smoothness of $f$. 
Written another way, we have
$$
B \leq C_p(\mu)\left( \Var(  \EE[ \nabla f(U)|X] )  + \left|\EE [\nabla f(U)] \right|^2 \right).
$$

By symmetry, we  obtain a similar bound with the roles of $\mu$ and $\nu$  (resp. $X$ and $Y$) reversed.  Hence, taking a convex combination of these two separate bounds, we  find
\begin{align*}
\Var(f(U)) &\leq \frac{C_p(\nu)^2 + C_p(\mu)^2}{C_p(\mu)+C_p(\nu)}  \EE\left[ |\nabla f(U)|^2  \right]  
+   2  \frac{C_p(\nu) C_p(\mu)}{C_p(\mu)+C_p(\nu)}    \left|\EE [\nabla f(U)] \right|^2   \\
&+  \frac{C_p(\nu) C_p(\mu)}{C_p(\mu)+C_p(\nu)}  \Big(     \Var(  \EE[ \nabla f(U)|X] )     +    \Var(  \EE[ \nabla f(U)|Y] )     \Big).
\end{align*}
Applying \eqref{varianceprojection} to the sum of  variance terms, we conclude \eqref{needToEliminateVariance} upon simplification.

At this point, we need to deal with the term $\Var(\nabla f(U))$ in \eqref{needToEliminateVariance} in order to bound $C_p(\mu\star \nu)$.  To this end,  Lemma \ref{lem:VarBound} ensures the existence of a sequence  $(f_n)$  with  $\EE\left[| \nabla f_n(U)|^2 \right] = 1$, which satisfies 
$$%
\lim_{n\to\infty} \Var( f_n(U))   = C_p(\mu\star\nu)
$$%
and 
\begin{align}
\lim_{n\to\infty}\Var(\nabla f_n(U) ) \geq \frac{(C_p(\mu\star\nu) - \sigma^2)^2}{(C_p(\mu\star\nu) - \sigma^2)^2 + C_p(\mu\star\nu) \sigma^2}, \label{varianceLB}
\end{align}
where $\sigma^2 :=\sigma^2(\mu\star\nu)$ as in the statement of the theorem.  Substituting this sequence into \eqref{needToEliminateVariance} and bounding the variance terms with \eqref{varianceLB} completes the proof.

\end{proof}
\begin{remark}
The above proof strategy does not appear to easily extend show stability of \eqref{eq:ConvolutionSubset}.  The reason is that the separate bounds on the quantities $A$ and $B$ cause  both Poincar\'e constants $C_p(\mu)$ and $C_p(\nu)$ to appear in bounding $C_p(\mu\star\nu)$.  
\end{remark}
\begin{remark} We remark that the above proof closely follows the development in \cite{johnson2004convergence}.  The key difference is that our use of the variance inequality \eqref{varianceprojection} avoids the introduction of Fisher information as in Johnson's proof. 
\end{remark}

At this point, we need only to prove Theorem \ref{thm:CpConverge}.  
\begin{proof}[Proof of Theorem \ref{thm:CpConverge}]
Starting with Corollary \ref{cor:IIDstable}, we have  
\begin{align*}
C_p(\nu_2)  &\leq  C_p(\nu_1)    - \frac{C_p(\nu_1) }{4 }  \frac{(C_p(\nu_2) - 1)^2}{(C_p(\nu_2) - 1)^2 + C_p(\nu_2)  } \\
&=C_p(\nu_1)\left(1 - \frac{1}{4}   \frac{(C_p(\nu_2) - 1)^2}{(C_p(\nu_2) - 1)^2 + C_p(\nu_2)  } \right) 
\end{align*}
since $\sigma^2(\nu_2)=1$ due to the isotropic assumption.  On rearranging, we find 
\begin{align}
C_p(\nu_1) - C_p(\nu_2) &\geq \frac{   C_p(\nu_2)  (C_p(\nu_2) - 1)^2  }{  3(C_p(\nu_2) - 1)^2 + 4 C_p(\nu_2)     } \notag \\
&\geq
\begin{cases}
\frac{1}{6}(C_p(\nu_2) - 1)^2 & 1 \leq C_p(\nu_2)< 2\\
\frac{1}{3}(C_p(\nu_2) - 1) - 1/6 & C_p(\nu_2)\geq 1,
\end{cases} \label{polyLB}
\end{align}
where the second inequality follows from elementary calculus.
Using the linear lower bound in \eqref{polyLB},  a straightforward inductive argument gives 
$$
C_p(\nu_{2^n})-1 \leq \frac{1}{2}\left(1-\left(3 \over 4\right)^n \right) + \left(3 \over 4\right)^n (C_p(\nu_1)-1),
$$
which is \eqref{eq:CpConverge1}.  

To establish \eqref{eq:CpConverge2}, construct a sequence $(a_n)_{n\geq 0}$ inductively starting with $a_0=2$, and defining $a_{n+1}$ to be the positive root of the quadratic equation
\begin{align}
a_{n+1} + \frac{1}{6}(a_{n+1} - 1)^2  = a_n, ~~~~n\geq 0. \label{quadRecurrence}
\end{align}
If $1\leq C_p(\nu_{1}) \leq 2$, the quadratic bound  in \eqref{polyLB} implies 
$$
C_p(\nu_{2^{n+1}}) + \frac{1}{6}(C_p(\nu_{2^{n+1}}) - 1)^2  \leq  C_p(\nu_{2^{n}})
$$
for $n\geq 0$, so that we necessarily  have $C_p(\nu_{2^{n}})\leq a_n$ for all $n\geq 0$.  Hence, we only need to upper bound the sequence $(a_n)_{n\geq 0}$. To that end, applying the quadratic formula to \eqref{quadRecurrence}, we have 
$$
a_{n+1} = -2 + \sqrt{9 + 6(a_{n}-1) }, ~~~~n\geq 0.
$$
We claim that $a_n \leq 1 + \frac{7}{n+7}$.  Indeed, this is true for $n=0$ by definition.  By induction, 
$$
a_{n+1} = -2 + \sqrt{9 + 6(a_{n}-1) } \leq -2 + \sqrt{9 + 6\frac{7}{n+7} } \leq  1 + \frac{7}{(n+1)+7},
$$
where the last inequality can be checked to hold for all $n\geq -1$.  This completes the proof. 
\end{proof}
 
 \subsection*{Acknowledgments}
 This work was supported in part by the France-Berkeley fund and NSF CCF-1750430.  The author thanks Max Fathi, Oliver Johnson and Michel Ledoux for  comments and discussions. We also would like to acknowledge the hospitality of the Modern Mathematical Methods for Data Analysis workshop held in Li\`ege Universit\'e, and the Workshop on Stability of Functional Inequalities and Applications at Universit\'e Paul Sabatier, Institut de Math\'ematiques de Toulouse.

%

\end{document}